\documentclass[reqno,11pt]{amsart} \numberwithin{equation}{section}
\input{amssym.def}
\input{amssym.tex}
\begin{document}
\title[ Strong Convergence Theorems for Equilibrium problems...] { Strong convergence theorems  by a new hybrid method for equilibrium problems and relatively nonexpansive mappings in Banach spaces}
\author[S. Alizadeh and F. Moradlou]{Sattar Alizadeh$^1$ and Fridoun
Moradlou$^2$}
\address{\indent $^{1,2}$Department of Mathematics
\newline \indent Sahand University of Technology
\newline \indent Tabriz, Iran}
\email{\rm $^1$ sa\_alizadeh@sut.ac.ir} \email{\rm $^2$ moradlou@sut.ac.ir \& fridoun.moradlou@gmail.com}
\thanks{}
\begin{abstract}
In this paper, we introduce a new modified Ishikawa iteration for
finding a common element of the set of solutions of an equilibrium
problem and the set of fixed points of  relatively nonexpansive mappings
in a Banach space. Our results generalize, extend and enrich some
existing results in the literature.
\end{abstract}
\subjclass[2010]{Primary 47H10, 47H09, 47J25, 47J05}

\keywords{Equilibrium problems, Fixed point, Hybrid method, Strong convergence, Weak convergence}

\maketitle

\baselineskip=15.8pt \theoremstyle{definition}
  \newtheorem{df}{Definition}[section]
    \newtheorem{rk}[df]{Remark}
\theoremstyle{plain}
 \newtheorem{lm}[df]{Lemma}
  \newtheorem{thm}[df]{Theorem}
  \newtheorem{exa}[df]{example}
  \newtheorem{cor}[df]{Corollary}
\setcounter{section}{0}
\section{Introduction}
Throughout this paper,
let $E$ be a real Banach space with the dual space $E^{*}$ and let
$C$ be a nonempty closed  convex subset of $E$. Let $f$ be a
bifunction from $C\times C$ to $\mathbb{R}$. The equilibrium problem
for $f:C\times C\rightarrow\mathbb{R}$ is to find $x\in E$ such that
\begin{equation}\label{ep.1}
f(x,y)\geq0,\quad(y\in C).
\end{equation}
The set of solutions of (\ref{ep.1}) is denoted by $EP(f)$, i.e.,
$$EP(f)=\{x\in C:\;f(x,y)\geq0,\;\;\forall y\in C\}.$$
\par
 A self-mapping  $S$ of $C$ is called nonexpansive if
$$\|Sx-Sy\|\leq\|x-y\|,\quad( x,y\in C).$$
We denote by $F(S)$ the set of fixed points of $S$.
\par
Let $S:C\rightarrow E^{*}$ be a mapping and let $f(x,y)=\langle
Sx,y-x\rangle$ for all $x,y\in C$. Then $z\in EP(f)$ if only if
$\langle Sz,y-z\rangle\geq0$ for all $y\in C$, i.e., $z$ is a solution
of the variational inequality $\langle Sx,y-x\rangle\geq0$. So, the
formulation (\ref{ep.1}) includes variational inequalities as
special cases. Also, numerous problems in physics, optimization and
economics reduce to find a solution of (\ref{ep.1}). Some methods
have been proposed to solve the equilibrium problem; see for
instance \cite{b,a4,e,a16,Kinder}.
\par
In the recent years, many authors studied the problem of finding a
common element of the set of fixed points of a nonexpansive mapping
and the set of solutions of an equilibrium problem in the framework
of Hilbert spaces and Banach spaces, respectively; see for instance,
\cite{a2,a3,a5,a9,a10,a18,a20,a23,a24} and the references therein.
\par
\par
Let $C$ be a nonempty closed convex subset of a Banach space. In
1953, for a self-mapping $S$ of $C$, Mann \cite{mann} defined the
following iteration procedure:
\begin{equation}\label{mann}
\begin{cases}
 & x_{0} \in E \,\,\,\hbox{chosen arbitrarily},\\
 & x_{n+1}  = \alpha_{n}x_{n}+(1-\alpha_{n})Sx_{n},
\end{cases}
\end{equation}
where $0\leq \alpha_{n}\leq 1$ for all $n \in \mathbb{N}\cup \{0\}$,
\par
Let $K$ be a compact convex subset of a Hilbert space $H$.
 In 1974, for a Lipschitzian pseudocontractive self-mapping $S$ of $K$, Ishikawa \cite{ishikawa}
defined the following iteration procedure:
\begin{equation}\label{ishikawa}
\begin{cases}
  & x_{0} \in K \,\,\,\hbox{chosen arbitrarily},\\
  & y_{n}= \beta_{n}x_{n}+ (1-\beta_{n})Sx_{n}, \\
  & x_{n+1}  = \alpha_{n}x_{n}+(1-\alpha_{n})Sy_{n},
\end{cases}
\end{equation}
where $0\leq \beta_{n} \leq \alpha_{n}\leq 1$ for all $n \in
\mathbb{N}\cup \{0\}$ and he proved strong convergence of the
sequence $\{x_{n}\}$ generated by the above iterative scheme if
$\lim_{n \to \infty} \beta_{n} =1$ and $\sum_{n=1}^{\infty}
(1-\alpha_{n})(1-\beta_{n}) = \infty$. By taking $\beta_{n} =1$ for
all $n \geq 0$ in $(\ref{ishikawa})$, Ishikawa iteration process
reduces to Mann iteration process.
\par
In general, to gain the convergence in Mann and Ishikawa iteration
processes, we must assume that underlying space $E$ has elegant
properties. For example, Reich \cite{Reich} proved that if $E$ is a
uniformly convex Banach space with a Fr\'{e}chet differentiable norm
and if $\{\alpha_{n}\}$ is such that $\sum_{n=1}^{\infty} \alpha_{n}
(1- \alpha_{n}) = \infty$, then the Mann iteration scheme converges
weakly to a fixed point of $T$. However, we know that the Mann
iteration process is weakly convergence even in a Hilbert space
\cite{Genel}. Also, Tan and Xu \cite{Tan} proved that if $E$ is a  a
uniformly convex Banach space which satisfies Opial's condition or
whose norm is Fr\'{e}chet differentiable and if $\{\alpha_{n}\}$ and
$\{\beta_{n}\}$ are such that $\sum_{n=1}^{\infty} \alpha_{n} (1-
\alpha_{n})$ diverges, $\sum_{n=1}^{\infty} \alpha_{n} (1-
\beta_{n}) $ converges and $\limsup \beta_{n} < 1$, then Ishikawa
iteration process converges weakly to a fixed point of $T$.
\par
It easy to see that process (\ref{ishikawa}) is more general than
the process (\ref{mann}). Also, for a Lipschitz pseudocontractive
mapping in a Hilbert space, process (\ref{mann}) is not known to
converge to a fixed point while the process (\ref{ishikawa}) is
convergence. In spite of these facts, researchers are interested to
study the convergence theorems by process (\ref{mann}), because of
the formulation of process (\ref{mann}) is simpler than that of
(\ref{ishikawa}) and if $\{\beta_{n}\}$ satisfies suitable
conditions, we can gain a convergence theorem for process
(\ref{ishikawa}) on a convergence theorem for process (\ref{mann}).
\par
In recent years, many authors have proved weak or strong convergence
theorems for some nonlinear mappings by using various iteration
processes in the framework of Hilbert spaces and Banach spaces, see,
\cite{am1, am2, Nakajo, taka2003, T5}.
\par
Recently, to obtain strong convergence, many mathematicians have
been extensively considered modified processes. Nakajo and Takahashi
\cite{Nakajo} proposed the following modification  of the Mann's
iteration for a nonexpansive self-mapping $S$ of a nonempty, closed convex subset  $C$ in a Hilbert
space $H$:
\begin{equation*}\label{mod. mann}
\begin{cases}
  & x_{0} \in C \,\,\,\hbox{chosen arbitrarily},\\
  &y_{n}  = \alpha_{n}x_{n}+ (1-\alpha_{n})Sx_{n}, \\
  &C_{n}  = \{u \in C\,\,:\|y_{n}-u\|\leq \|x_{n}-u\|\},\\
  &Q_{n} = \{u \in C\,\,:\langle x_{n}-u, x_{0}-x_{n}\rangle \geq
  0\},\\
  &x_{n+1} = P_{C_{n}\cap Q_{n}}x_{0},
\end{cases}
\end{equation*}
where $P_{K}$ denotes the metric projection from $H$ onto a closed
convex subset $K$ of $H$. They proved strong convergence of the
sequence $\{x_{n}\}$, if the sequence $\{\alpha_{n}\}$ is bounded
above from one.
\par
In 2005, Matsushita and Takahashi \cite{Mat} introduced a new hybrid algorithim for a relatively self-mapping $S$ of  $C$ in a Banach space $E$ as follows:
\begin{equation*}
\begin{cases}
 & x_{0}=x \in C\,\,\,\hbox{chosen arbitrarily},\\
 & y_n=J^{-1}(\alpha_nJx_n+(1-\alpha_n)JSx_n),\\
 &H_n=\lbrace u\in C:\phi(u,y_n)\leq\phi(u,x_n)\rbrace,\\
 &W_n=\lbrace z\in C:\langle x_n-z,Jx_n-Jx\rangle\leq0\rbrace,\\
 &x_{n+1}=\Pi_{C_n\cap Q_n}x, n=0,1,2,...,
  \end{cases}
\end{equation*}
where $J$ is the duality mapping on $E$ and $\Pi_{F(S)}$ is the generalized projection from $C$ onto $F(S)$. They proved that $\{x_n\}$ converges strongly to $\Pi_{F(S)}x$.
\par
In 2006, Martinez-Yanes and Xu \cite{martinez} introduced the
following modified Ishikawa iteration process for a nonexpansive
self-mapping $S$ of a nonempty, closed convex subset $C$ with $F(S) \neq\emptyset$ in a Hilbert space
$H$:
\begin{equation*}\label{mod. ishikawa}
\begin{cases}
  & x_{0} \in E\,\,\,\hbox{chosen arbitrarily},\\
  & z_{n} = \beta_{n}x_{n}+(1-\beta_{n})Sx_{n}\\
  &y_{n}  = \alpha_{n}x_{n}+ (1-\alpha_{n})Sz_{n}, \\
  &C_{n}  = \{u \in C\,\,:\|y_{n}-u\|^2\leq \|x_{n}-u\|^2 \\
  &\hspace{3cm} + (1-\alpha_{n})(\|z_{n}\|^{2} - \|x_{n}\|^{2} +2 \langle x_{n}-z_{n},u\rangle \geq 0 )\},\\
  &Q_{n} = \{u \in C\,\,:\langle x_{n}-u, x_{0}-x_{n}\rangle \geq
  0\},\\
  &x_{n+1} = P_{C_{n}\cap Q_{n}}x_0,
\end{cases}
\end{equation*}
where $\{\alpha_{n}\}$ and $\{\beta_{n}\}$ are sequences in $[0,1]$.  They proved that if $\{\alpha_{n}\}$ bounded above from one
and $\lim_{n \to \infty}\beta_{n} = 1$, then the sequence
$\{x_{n}\}$ converges strongly to $P_{F(S )}x_{0}$.
\par
In 2007, Tada and Takahashi \cite{a27} for finding an element of $EP(f)\cap F(S)$, introduced the following iterative scheme by using the hybrid projection method for a nonexpansive self-mapping $S$ of a nonempty, closed convex subset $C$ in a Hilbert space:
\begin{equation*}
\begin{cases}
&x_{0} \in H\,\,\,\hbox{chosen arbitrarily},\\
 &u_n\in E\;\text{such that}\quad f(u_n,y) +\frac{1}{r_n}\langle y-u_n,u_n-x_n\rangle\geq0,\quad\forall\: y\in E\\
 & w_n=\alpha_n x_n+(1-\alpha_n)Su_n,\\
 &C_n=\lbrace u\in H:\|w_n-u\|\leq\|x_n-u\|\rbrace,\\
 &Q_n=\lbrace u\in H:\langle x_n-u,x_n-x\rangle\leq0\rbrace,\\
 &x_{n+1}=P_{C_n\cap Q_n}x,
  \end{cases}
  \end{equation*}
  for all $n\in\mathbb{N}\cup\{0\}$, where $\{\alpha_n\}\subset[a,b]$ for some $a,b\in(0,1)$ and $\{r_n\}\subset(0,\infty)$ satisfies $\liminf_{n\rightarrow\infty}r_n>0.$ Thus, they proved that $\{x_n\}$ and  $\{u_n\}$ converge strongly to $u$, where
  \linebreak
   $u=P_{F(S)\cap EP(f)}x_0\in F(S)\cap EP(f).$
   \par
   In 2009, Takahashi and Zembayashi \cite{T4}  presented a new hybrid iterative method for finding an element of $EP(f)\cap F(S)$,  by using the hybrid projection method for relatively nonexpansive self-mapping $S$ of $C$ in a Banach space $E$:
\begin{equation*}
\begin{cases}
&x_{0} \in C\,\,\,\hbox{chosen arbitrarily},\\
& y_n=J^{-1}(\alpha_n Jx_n+(1-\alpha_n)JSx_n),\\
 &u_n\in C\;\text{such that}\quad f(u_n,y) +\frac{1}{r_n}\langle y-u_n,Ju_n-Jy_n\rangle\geq0,\quad\forall\: y\in C\\
 &H_n=\lbrace u\in C:\|w_n-u\|\leq\|x_n-u\|\rbrace,\\
 &W_n=\lbrace u\in C:\langle x_n-u,Jx_n-Jx\rangle\leq0\rbrace,\\
 &x_{n+1}=P_{H_n\cap W_n}x,
  \end{cases}
  \end{equation*}
  for all $n\in\mathbb{N}\cup\{0\}$, where $J$ is the duality mapping on $E$, $\{\alpha_{n}\}$ is a sequence in $[0,1]$ such that $\lim_{n \to \infty}\alpha_{n}(1-\alpha_{n})>0$ and $\{r_n\}$ is a sequence in $[a,\infty)$  for some $a>0$. They proved that the sequence $\{x_{n}\}$ converges strongly to $\Pi_{F(S )\cap EP(f)}x$, where $\Pi_{F(S )\cap EP(f)}$ is the generalized projection from $C$ onto $F(S)\cap EP(f)$.
  \par
In this paper, employing the idea of Nakajo-Takahashi \cite{Nakajo} and  Takahashi-Zembayashi \cite{T4}, we modify Ishikawa iteration process for finding a common element of the set of solution of an equilibrium problem and the set of fixed points of a relatively nonexpansive mapping by applying the hybrid projection method in a Banach space.
\section{Preliminaries}
  Let $E$ be a real Banach space with $\|.\|$ and dual space $E^{*}$. We denote the weak convergence and the strong convergence of $\{x_n\}$ to $x\in E$ by $x_n\rightharpoonup x$ and $x_n\rightarrow x$, respectively and denote by $J$ the normalized duality mapping from $E$ into $2^{E^{*}}$ defined by
 $$Jx=\{x^{*}\in E^{*}:\; \langle x,x^{*}\rangle=\|x\|^2=\|x^{*}\|^2\}$$
 for all $x\in E$, where $\langle.,.\rangle$ denotes the generalized duality pairing between $E$ and $E^{*}$. A Banach space $E$ is said to be strictly convex if  $\|\frac{x+y}{2}\|<1$ for all
  $x,y\in E$ with $\|x\|=\|y\|=1$ and $x\neq y.$ It is also said to be uniformly convex if for every $\epsilon\in (0,2]$,
  there exists a $\delta>0$, such that $\|\frac{x+y}{2}\|<1-\delta$ for all $x,y\in E$ with $\|x\|=\|y\|=1$ and
  $\|x-y\|\geq\epsilon$. Furthermore, $E$ is called smooth if the limit
 \begin{equation}\label{eq1}
 \lim_{t\rightarrow0}\frac{\|x+ty\|-\|x\|}{t}
 \end{equation}
 exists for all $x,y\in B_{E}=\{x\in E:\;\|x\|=1\}$. It is also said to be uniformly smooth if the limit
  $(\ref{eq1})$ is attained uniformly for all $x,y\in E$. A Banach space $E$ is said to have the Kadec--Klee property if for every sequence $\{x_n\}$ in $E$ that $x_n\rightharpoonup x$ and $\|x_n\|\rightarrow\|x\|$ then $x_n\rightarrow x$. It is known that if $E$ is uniformly convex, then $E$ has the Kadec--Klee property. Also, $E$ is uniformly convex if and only if $E^*$ is uniformly smooth. For more details see \cite{Agarwal, T3}.
\par
 Many properties of the normalized duality mapping $J$ have been given in \cite{Agarwal, T3}. We give some of those in the following:
  \begin{enumerate}
\item $J(0) =\{0\}$.
\item For every $x\in E$, $Jx$ is nonempty closed convex and bounded subset of $E^*$.
\item  If $E^*$ is strictly convex, then $J$ is single-valued.
\item If $E$ is strictly convex, then $J$ is one-one, i.e., if $x\neq y$ then $ Jx\cap Jy=\phi$.
\item If $E$ is reflexive, then $J$ is onto.
\item If $E$ is smooth, then $J$ is single-valued.
\item if $E$ is smooth and reflexive, then $J$ is norm-to-weak continuous, that is, $Jx_n\rightharpoonup Jx$ whenever $x_n\rightarrow x$.
\item If $E$ is smooth, strictly convex  and reflexive  and $J^*:E^*\rightarrow 2^E$ is the normalized duality mapping on $E^*$, then $J^{-1}=J^*$, $JJ^*=I_{E^*}$ and $J^*J=I_{E}$, where $I_{E}$ and $I_{E^*}$ are the identity mapping on $E$ and $E^*$, respectively.
\item If $E$ is  uniformly convex and uniformly smooth, then $J$ is uniformly norm-to-norm continuous on bounded sets of $E$ and $J^{-1}=J^*$ is also uniformly norm-to-norm continuous on bounded sets of $E^*$.
\end{enumerate}
\par
Let C be a nonempty, closed convex subset  of a smooth, strictly convex and  reflexive  Banach space $E$. We denote by $\phi$ the function $\phi:E\times E\rightarrow\mathbb{R}$ defined as follows:
$$\phi(x,y)=\|x\|^2-\langle x,Jy\rangle+\|y\|^2,$$
for all $x,y\in E$. It is clear from the definition of the function $\phi$ that for all $x,y,z\in E,$
\begin{enumerate}
\item $(\|y\|-\|x\|)^2\leq\phi(x,y)\leq(\|y\|+\|x\|)^2$,
\item $\phi(x,y)=\phi(x,z)+\phi(z,y)+2\langle x-z,Jz-Jy\rangle.$
\end{enumerate}
Observe that if $E$ is in a Hilbert space then $\phi(x,y)=\|x-y\|^2$.
\par
In 1996, Alber \cite{Alber}, defined the generalized projection mapping as follows:
\begin{df}
Let C be a nonempty, closed convex subset  of a smooth, strictly convex and  reflexive  Banach space $E$.  The generalized projection $\Pi_C:E\rightarrow C$ is a mapping that assigns to an arbitrary point $x\in E$ the minimum point of the functional $\phi(x, y)$,  i.e., $\Pi_C x=x_0$, where $x_0$ is the solution to the minimization problem
$$\phi(x_0, x)=\inf_{y\in C}\phi(y, x).$$
\end{df}
Existence and uniqueness of the operator $\Pi_C$ follows from the properties of the functional  $\phi(x, y)$ and strict monotonicity of the mapping $J$.
\par
Let $S$ be a self-mapping of $C$. A point $p$ in $C$ is said to be an asymptotic fixed point of $S$ \cite{Reich}, if there exists a sequence $\{x_n\}$ in $C$ such that $x_n\rightharpoonup p$ and $\|x_n-Sx_n\|\rightarrow0$. we denote by $\hat{F}(S)$ the set of all asymptotic fixed points of $S$. A self-mapping $S$ of $C$ is said to be relatively nonexpansive [5-7], if the following conditions are satisfied:
\begin{enumerate}
\item $F(S)$ is nonempty;
\item $\phi(u,Sx)\leq\phi(u,x)$,  $\forall u\in F(S)$, $\forall  x\in C$;
\item  $F(S)=\hat{F}(S)$.
\end{enumerate}
\begin{lm}\cite{Mat}
Let $C$ be a nonempty, closed convex subset of a smooth, strictly convex and reflexive Banach space $E$, and let
$S$ be a relatively nonexpansive self-mapping of $C$. Then $F(S)$ is closed and convex.
\end{lm}
\begin{lm}\cite{Alber}\label{lm.1}
Let $C$ be a nonempty, closed convex subset of smooth, strictly convex and reflexive Banach space $E$. Then
$$\phi(x, \Pi_{C}y)+\phi(\Pi_{C}y,y)\leq\phi(x,y),$$
for all $x\in C$ and all $y\in E$.
\end{lm}
\begin{lm}\cite{Alber}\label{lm.2}
Let $C$ be a nonempty, closed convex subset of smooth, strictly convex and reflexive Banach space $E$, $x\in E$ and $z\in C$. Then
$$z=\Pi_{C}x\Longleftrightarrow\langle y-z, Jx-Jz\rangle\leq0,$$
for all $y\in C$.
\end{lm}
\begin{lm}\cite{Kam}\label{lm.4}
Let $E$ be a smooth and uniformly convex Banach space and let $\{x_n\}$ and $\{y_n\}$ be sequences in $E$ such that either $\{x_n\}$ or $\{y_n\}$  is bounded. If $lim_{n\rightarrow\infty}\phi(x_n,y_n)=0$, then $lim_{n\rightarrow\infty}\|x_n-y_n\|=0$.
\end{lm}
\begin{lm}\cite{Xu}\label{lm.5}
Let $E$ be a uniformly convex Banach space and  $r > 0$. Then there exists a strictly increasing, continuous and convex function $g: [0,2r]\rightarrow\mathbb{R}$ such that $g(0)=0$ and
$$\|tx+(1-t)y\|^2\leq t\|x\|^2+(1-t)\|y\|^2-t(1-t)g(\|x-y\|),$$
for all $x,y \in B_r$ and $t\in[0,1]$, where $B_r=\{z\in E:\|z\|\leq r\}.$
\end{lm}
To study the equilibrium problem, for the bifunction $f:C\times C\rightarrow\mathbb{R}$, we assume that $f$ satisfies the following conditions:
\begin{enumerate}
\item[(A1)] $f(x,x)=0$ for all $x\in C;$
\item[(A2)]$f$ is monotone, i.e., $f(x,y)+f(y,x)\leq0$ for all $x,y\in C;$
\item[(A3)]for each $x,y,z\in C$,
$$\hspace{-4cm}
\lim_{t\downarrow0} f(tz+(1-t)x,y)\leq f(x,y);$$
\item[(A4)]for each $x\in C,\;y\mapsto f(x,y)$ is convex and lower semicontinuous.
\end{enumerate}
\par
The following lemma can be found in \cite{b}.
\begin{lm}\label{lm.6}
Let $C$ be a nonempty, closed convex subset of a smooth, strictly convex and reflexive Banach space $E$, $f$ be a bifunction from $C\times C$ to $\mathbb{R}$ satisfying $(A1)-(A4)$ and let $r > 0$ and $x\in E$. Then, there exists
$z\in C$ such that
$$f (z,y)+\frac{1}{r}\langle y-z,Jz-Jx\rangle\geq0,$$
for all $y\in C.$
\end{lm}
\par
\begin{lm}\cite{T4}\label{lm.7}
 Let $C$ be a nonempty, closed convex subset of a smooth, strictly convex and reflexive Banach space $E$,  $f$ be a bifunction from $C\times C$ to $\mathbb{R}$ satisfying $(A1)-(A4)$ and let $r > 0$ and $x\in E$. Define a mapping $T_r :E\rightarrow C$ as follows:
$$\hspace{-2cm}T_r (x) = \{z\in C :f (z,y)+\frac{1}{r}\langle y-z,Jz-Jx\rangle\geq0,\; \forall y\in C\},$$
for all $x\in E$. Then, the following statements hold:

    \begin{enumerate}
      \item[(i)]$T_r$ is singel-valued;
      \item[(ii)]$T_r$ is firmly nonexpansive-type, i.e., for all $x,y\in E$,
      $$\langle T_rx-T_ry,JT_rx-JT_ry\rangle\leq\langle T_rx-T_ry,Jx-Jy\rangle;$$
      \item[(iii)]$F(T_r)=EP(f);$
      \item[(iv)]EP(f) is closed and convex and $T_r$ is relatively nonexpansive mapping.
    \end{enumerate}
    \end{lm}
\section{Main Results}
In this section, we prove the strong convergence theorem for finding a common element of the set of solution of an equilibrium problem and the set of fixed points of a relatively nonexpansive  mapping.
 \begin{thm}\label{thm1}
 Let $C$ be a nonempty closed convex subset of uniformly smooth and uniformly convex Banach space $E$. Let $f$ be a bifunction from $C\times C$ to $\mathbb{R}$ satisfying $(A1)-(A4)$ and $S$ be a relatively nonexpansive self-mapping of $C$ with $F(S)\cap EP(f)\neq\phi$.
Assume that $0< a\leq\alpha_n\leq1$ and $\{r_n\}\subset(0,\infty)$ satisfies
$\liminf_{n\rightarrow\infty}r_n>0$ and  $\{\beta_n\}$ is sequence
in $[0,1]$ such that $\liminf_{n\rightarrow\infty}\beta_n(1-\beta_n)>0$. If $\{x_n\}$ and $\{u_n\}$ be sequences generated by
$x=x_1\in C$ and
\begin{equation*}
\begin{cases}
& z_n=J^{-1}(\beta_nJx_n+(1-\beta_n)JSx_n),\\
&u_n\in E\;\text{such that}\quad f(u_n,y) +\frac{1}{r_n}\langle y-u_n,Ju_n-Jz_n\rangle\geq0,\quad\forall\: y\in E\\
 & y_n=J^{-1}(\alpha_nJz_n+(1-\alpha_n)Ju_n),\\
 &C_n=\lbrace v\in C:\phi(v,y_n)\leq\phi(v,x_n)\rbrace,\\
 &Q_n=\lbrace z\in C:\langle x_n-z,Jx_n-Jx\rangle\leq0\rbrace,\\
 &x_{n+1}=\Pi_{C_n\cap Q_n}x,
  \end{cases}
\end{equation*}
for all $n\in\mathbb{N}$. Then, $\{x_n\}$ converges strongly to $\Pi_{F(S)\cap EP(f)}x$, where $\Pi_{F(S)\cap EP(f)}$ is the generalized projection of $E$ onto $F(S)\cap EP(f)$.
\end{thm}
\begin{proof}
First, we show that $C_n\cap Q_n$  is closed and convex. It is easily seen that $C_n$ is closed and  $Q_n$ is closed and convex for all $n\in \mathbb{N}$. Since
$$\phi(v,y_n)\leq\phi(v,x_n)\Longleftrightarrow \|y_n\|^2-\|x_n\|^2-2\langle v,Jy_n-Jx_n\rangle\geq0,$$
for all $v\in C_n$. Then $v\in C_n$ is convex. So $C_n\cap Q_n$  is
closed and convex for all $n\in \mathbb{N}$. Let $u\in F(S)\cap
EP(f)$. From $u_n=T_{r_n}z_n$ and using Lemma \ref{lm.7}(iv), we can
conclude that $T_{r_n}$ are relatively nonexpansive, in addition,
$S$ is relatively nonexpansive, so by the convexity of $\|.\|^2$ we
get
\begin{equation}\begin{aligned}\label{inq.2}
\phi(u,u_n)&=\phi(u,T_{r_n}z_n)\\
&\leq\phi(u,z_n)\\
&=\phi(u,J^{-1}(\beta_nJx_n+(1-\beta_n)JSx_n)\\
&=\|u\|^2-2\langle u,\beta_nJx_n+(1-\beta_n)JSx_n\rangle +\|\beta_nJx_n+(1-\beta_n)JSx_n\|^2\\
&\leq\|u\|^2-2\beta_n\langle u,Jx_n\rangle -2(1-\beta_n)\langle u,JSx_n\rangle+\beta_n\|x_n\|^2+(1-\beta_n)\|Sx_n\|^2\\
&=\beta_n\phi(u,x_n)+(1-\beta_n)\phi(u,Sx_n)\\
&\leq\phi(u,x_n),
\end{aligned}
\end{equation}
and therefore
\begin{equation}\begin{aligned}\label{}
\phi(u,y_n)&=\phi(u,J^{-1}(\alpha_nJz_n+(1-\alpha_n)Ju_n)\\
&=\|u\|^2-2\langle u,\alpha_nJz_n+(1-\alpha_n)Ju_n\rangle +\|\alpha_nJz_n+(1-\alpha_n)Ju_n\|^2\\
&\leq\|u\|^2-2\alpha_n\langle u,Jz_n\rangle -2(1-\alpha_n)\langle u,Ju_n\rangle+\alpha_n\|z_n\|^2+(1-\alpha_n)\|u_n\|^2\\
&=\alpha_n\phi(u,z_n)+(1-\alpha_n)\phi(u,u_n)\\
&\leq\alpha_n\phi(u,x_n)+(1-\alpha_n)\phi(u,x_n)\\
&=\phi(u,x_n),
\end{aligned}
\end{equation}
So $u\in C_n$ for all $n\in \mathbb{N}$. Hence,
$$F(S)\cap EP(f)\subset C_n.$$
\par
Now, using induction, we will show  that $F(S)\cap EP(f)\subset C_n\cap Q_n$ for all $n\in\mathbb{N}$. For $n=1$, we have $x_1=x\in C$
   and $Q_1=C$, therefore $ F(S)\cap EP(f)\subset Q_1$ and hence  $F(S)\cap EP(f)\subset C_1\cap Q_1$. Assume that  $F(S)\cap EP(f)\subset C_n\cap Q_n$  for some $n$.
Since $x_{n+1}=\Pi_{ C_n\cap Q_n}x$, by Lemma \ref{lm.2}, we have
\begin{equation}\label{inq.4}
\langle x_{n+1}-z,Jx-Jx_{n+1}\rangle\geqslant0,
\end{equation}
for all $z\in C_n\cap Q_n.$ By the induction assumption, we have
$F(S)\cap EP(f)\subset C_n\cap Q_n$, so we can conclude that for all $u\in
F(S)\cap EP(f)$  the inequality $(\ref{inq.4})$  holds. On the other hand, by
the definition of $Q_{n+1}$,  we obtain $F(S)\cap EP(f)\subset Q_{n+1}$.
Therefore,\linebreak
  $F(S)\cap EP(f)\subset C_{n+1}\cap Q_{n+1}$. So $\lbrace
x_n\rbrace$ is well-defined.
\par
Notice that the definition of $Q_n$
implies $x_n=\Pi_{Q_n}x.$ By Lemma \ref{lm.1} we get
\begin{equation}\label{inq.5}
\phi(x_n,x)=\phi(\Pi_{Q_n}x,x)\leq\phi(u,x)-\phi(u,\Pi_{Q_n}x)\leq\phi(u,x),
\end{equation}
 for all $ u\in F(S)\cap EP(f)\subset Q_n$.
   This yields that $\lbrace \phi(x_n,x)\rbrace$ is bounded. Therefore
   $\{x_n\}$, $\{Sx_n\}$, $\{y_n\}$ and $\{z_n\}$ are  bounded.
   \par
    Since $x_{n+1}=\Pi_{C_n\cap Q_n}x\in C_n\cap Q_n\subset Q_n,$ and $x_{n}=\Pi_{Q_n}x$, from the definition of   $\Pi_{Q_n}$  we have
\begin{equation}\label{inq.0}
\phi(x_n,x)\leq\phi(x_{n+1},x),
\end{equation}
for all $n\in\mathbb{N}$. It follows from $(\ref{inq.5})$  and $(\ref{inq.0})$, the sequence $\{\phi(x_n,x)\}$ is bounded and
\linebreak
 nondecreasing . So, $\lim_{n\rightarrow\infty}\{\phi(x_n,x)\}$ exists. by using Lemma \ref{lm.1} and $x_{n}=\Pi_{Q_n}x$, we also get
\begin{equation}\begin{aligned}\label{eq.1}
\phi(x_{n+1},x_n)&=\phi(x_{n+1},\Pi_{Q_n}x)\\
&\leq\phi(x_{n+1},x)-\phi(\Pi_{Q_n}x,x)\\
&=\phi(x_{n+1},x)-\phi(x_{n},x),
\end{aligned}\end{equation}
for all $n\in\mathbb{N}$. This yields that
\begin{equation}\label{inq.6}
\lim_{n\rightarrow\infty}\phi(x_{n+1},x_{n})=0.
\end{equation}
Since $x_{n+1}=\Pi_{C_n\cap Q_n}x\in C_n,$ we have
\begin{equation}\label{eq.2}
\phi(x_{n+1},y_n)\leq\phi(x_{n+1},x_n),
\end{equation}
for all $n\in\mathbb{N}$. from $(\ref{inq.6})$ and $(\ref{eq.2})$ we have
\begin{equation}\label{eq.3}
\lim_{n\rightarrow\infty}\phi(x_{n+1},y_{n})=0.
\end{equation}
From $(\ref{inq.6})$ and $(\ref{eq.3})$, since $E$ is uniformly convex and smooth, from Lemma \ref{lm.4} we get
$$\lim_{n\rightarrow\infty}\|x_{n+1}-x_{n}\|=\lim_{n\rightarrow\infty}\|x_{n+1}-y_{n}\|=0.$$
So, we obtain
\begin{equation}\label{eq.4}
\lim_{n\rightarrow\infty}\|x_{n}-y_{n}\|=0.
\end{equation}
Since $J$ is uniformly norm-to-norm continuous on bounded sets, from $(\ref{eq.4})$ we get
$$\lim_{n\rightarrow\infty}\|Jx_{n}-Jy_{n}\|=0$$
Let $r=\sup_{n\in\mathbb{N}}\{\|x_n\|,\|Sx_n\|\}$. Since $E$ is a uniformly smooth Banach space, we can conclude that $E^{*}$ is a uniformly convex Banach space. So, by Lemma \ref{lm.5} there exists a continuous, strictly increasing and convex function $g:[0,2r]\rightarrow\mathbb{R}$ whit $g(0)=0$ and for $u\in F(S)\cap EP(f)$ we have
\begin{equation}\begin{aligned}\label{inq.7}
\phi(u,z_n)&=\phi(u,J^{-1}(\beta_nJx_n+(1-\beta_n)JSx_n)\\
&=\|u\|^2-2\langle u,\beta_nJx_n+(1-\beta_n)JSx_n\rangle +\|\beta_nJx_n+(1-\beta_n)JSx_n\|^2\\
&\leq\|u\|^2-2\beta_n\langle u,Jz_n\rangle -2(1-\beta_n)\langle u,JSu_n\rangle+\beta_n\|z_n\|^2+(1-\beta_n)\|Su_n\|^2\\
&\hspace{1cm}-\beta_n(1-\beta_n)g(\|Jx_n-JSx_n\|)\\
&=\beta_n\phi(u,x_n)+(1-\beta_n)\phi(u,Sx_n)-\beta_n(1-\beta_n)g(\|Jx_n-JSx_n\|)\\
&\leq\phi(u,x_n)-\beta_n(1-\beta_n)g(\|Jx_n-JSx_n\|),
\end{aligned}
\end{equation}
and therefore
\begin{equation}\begin{aligned}\label{inq.8}
\phi(u,y_n)&=\phi(u,J^{-1}(\alpha_nJz_n+(1-\alpha_n)Ju_n)\\
&\leq\alpha_n\phi(u,z_n)+(1-\alpha_n)\phi(u,x_n)\\
&\leq\alpha_n\phi(u,x_n)-\alpha_n\beta_n(1-\beta_n)g(\|Jx_n-JSx_n\|)+(1-\alpha_n)\phi(u,x_n)\\
&\leq\phi(u,x_n)-\alpha_n\beta_n(1-\beta_n)g(\|Jx_n-JSx_n\|),
\end{aligned}
\end{equation}
hence
$$\alpha_n\beta_n(1-\beta_n)g(\|Jx_n-JSx_n\|)\leq\phi(u,x_n)-\phi(u,y_n).$$
Since $0<a\leq\alpha_n\leq1$, it is easy to see that
\begin{equation}\label{inq.9}
a\beta_n(1-\beta_n)g(\|Jx_n-JSx_n\|)\leq\phi(u,x_n)-\phi(u,y_n).
\end{equation}
On the other hand, we have
\begin{equation}\begin{aligned}\label{inq.10}
\phi(u,x_n)-\phi(u,y_n)&=\|x_n\|^2-\|y_n\|^2-2\langle u,Jx_n-Jy_n\rangle\\
&\leq\vert\|x_n\|^2-\|y_n\|^2\vert+2\vert\langle u,Jx_n-Jy_n\rangle\vert\\
&\leq\vert\|x_n\|-\|y_n\|\vert(\|x_n\|+\|y_n\|)+2\|u\|\|Jx_n-Jy_n\|\\
&\leq\vert\|x_n-y_n\|\vert(\|x_n\|+\|y_n\|)+2\| u\|\|Jx_n-Jy_n\|,
\end{aligned}\end{equation}
so
\begin{equation}\label{eq.5}
\hspace{-7cm}\lim_{n\rightarrow\infty}(\phi(u,x_n)-\phi(u,y_n))=0
\end{equation}
Since $\liminf_{n\rightarrow\infty}\beta_n(1-\beta_n)>0$, it follows from (\ref{inq.9}) that
\begin{equation*}
\lim_{n\rightarrow\infty}g(\|Jx_n-JSx_n\|)=0.
\end{equation*}
Hence, from property of $g$ we get
\begin{equation}\label{eq.6}
\lim_{n\rightarrow\infty}\|JSx_n-Jx_n\|=0.
\end{equation}
Since $J^{-1}$ is also uniformly norm-to-norm continuous on bounded sets, we have
\begin{equation}\label{eq.11}
\lim_{n\rightarrow\infty}\|x_n-Sx_n\|=0.
\end{equation}
Notice that $Jz_n-Jx_n=(1-\beta_n)(JSx_n-Jx_n)$, from $(\ref{eq.6})$ we can conclude that
\begin{equation*}
\lim_{n\rightarrow\infty}\|Jz_n-Jx_n\|=\lim_{n\rightarrow\infty}(1-\beta_n)\|JSx_n-Jx_n\|=0.
\end{equation*}
Since $J^{-1}$ is also uniformly norm-to-norm continuous on bounded sets, we obtain
\begin{equation}\label{eq.6.}
\lim_{n\rightarrow\infty}\|z_n-x_n\|=0.
\end{equation}
Since
$$\|y_n-z_n\|\leq\|y_n-x_n\|+\|x_n-z_n\|,$$
from $(\ref{eq.4})$, $(\ref{eq.6.})$ and last inequality we have
\begin{equation}\label{eq.7.}
\lim_{n\rightarrow\infty}\|y_n-z_n\|=0.
\end{equation}
Since $J$ is uniformly norm-to-norm continuous on bounded sets from (\ref{eq.7.}), we get
\begin{equation}\label{eq.8.}
\lim_{n\rightarrow\infty}\|Jy_n-Jz_n\|=0.
\end{equation}
Noticing that $Jy_n=\alpha_nJz_n+(1-\alpha_n)Ju_n$ and $\alpha_n<b<1$, we obtain
$$\|Jz_n-Ju_n\|=\frac{1}{1-\alpha_n}\|Jy_n-Jz_n\|\leq\frac{1}{1-b}\|Jy_n-Jz_n\|.$$
From (\ref{eq.8.}) and last inequality we get
\begin{equation}\label{eq.9.}
\lim_{n\rightarrow\infty}\|Jz_n-Ju_n\|=0.
\end{equation}
Since $J^{-1}$ is also uniformly norm-to-norm continuous on bounded sets, we get
\begin{equation}\label{eq.10.}
\lim_{n\rightarrow\infty}\|z_n-u_n\|=0.
\end{equation}
Since $\{x_n\}$ is bounded, there is a subsequence $\{x_{n_k}\}$ of $\{x_n\}$ such that $x_{n_k}\rightharpoonup\hat{x}$.
\par
Since $S$ is relatively nonexpansive, we can conclude from $(\ref{eq.11})$ that $\hat{x}\in\hat{F}(S)=F(S)$. We show that $\hat{x}\in EP(f)$.
From $x_{n_k}\rightharpoonup\hat{x}$, $(\ref{eq.6.})$ and $(\ref{eq.10.})$, we have $z_{n_k}\rightharpoonup\hat{x}$ and $u_{n_k}\rightharpoonup\hat{x}$.
\par
Since $\liminf_{n\rightarrow\infty}r_n>0$, from $(\ref{eq.9.})$, we derive that
\begin{equation}\label{eq.8}
\lim_{n\rightarrow\infty}\Big\|\frac{Ju_n-Jz_n}{r_n}\Big\|=\lim_{n\rightarrow\infty}\frac{1}{r_n}\|Ju_n-Jz_n\|=0.
\end{equation}
\par
 Since $u_n=T_{r_n}z_n$, we get
 \begin{equation}\label{inq.11}
f(u_n,z)+\frac{1}{r_n}\langle z-u_n,Ju_n-Jz_n\rangle\geq0,
\end{equation}
for all $y\in C$. Substituting $n$ by $n_k$ in $(\ref{inq.11}$) and by using condition $(A2)$, we obtain
$$\frac{1}{r_{n_k}}\langle z-u_{n_k},Ju_{n_k}-Jz_{n_k}\rangle\geq -f(u_{n_k},y)\geq f(y,u_{n_k}),$$
for all $y\in C$.  Letting $k\rightarrow\infty$, it follows from (\ref{eq.8}) and condition $(A4)$ that
$$0\geq f(y,\hat{x}),$$
for all $y\in C$. Suppose that $t\in(0,1]$, $y\in C$ and  $y_t=ty+(1-t)\hat{x}.$ Therefore, $y_t\in C$ and so $f(y_t,\hat{x})\leq0$. Hence
$$0=f(y_t,y_t)\leq tf(y_t,y)+(1-t)f(y_t,\hat{x})\leq tf(y_t,y),$$
and dividing by t, we have $f(y_t,y)\geq0$, for all $y\in E$. By taking the limit as $t\downarrow0$ and using $(A3)$, we get $\hat{x}\in EP(f)$.
\par
put $\omega=\Pi_{F(S)\cap EP(f)}x$. From $x_{n+1}=\Pi_{C_n\cap Q_n}x$ and $\omega\in F(S)\cap EP(f)\subset C_n\cap Q_n$, we have
$$\phi(x_{n+1},x)\leq\phi(\omega,x).$$
Since $\|.\|$ is weakly lower semicontinuous, we get
\begin{equation*}\begin{aligned}
\phi(\hat{x},x)=&\|\hat{x}\|^2-\langle\hat{x},Jx\rangle+\|x\|^2\\
&\leq\liminf_{n\rightarrow\infty}(\|x_{n_k}\|^2-\langle x_{n_k},Jx\rangle+\|x\|^2)\\
&=\liminf_{n\rightarrow\infty}\phi(x_{n_k},x)\\
&\leq\limsup_{n\rightarrow\infty}\phi(x_{n_k},x)\\
&\leq\phi(\omega,x).
\end{aligned}\end{equation*}
By the definition of $\Pi_{F(S)\cap EP(f)}$, we get $\hat{x}=\omega$, so we get $\lim_{n\rightarrow\infty}\phi(x_{n_k},x)=\phi(\omega,x).$
Hence we obtain
\begin{equation*}\begin{aligned}
0&=\lim_{n\rightarrow\infty}(\phi(x_{n_k},x)-\phi(\omega,x))\\
&=\lim_{n\rightarrow\infty}(\|x_{n_k}\|^2-\|\omega\|^2-2\langle x_{n_k}-\omega,Jx\rangle)\\
&=\lim_{n\rightarrow\infty}(\|x_{n_k}\|^2-\|\omega\|^2).
\end{aligned}\end{equation*}
Since $E$ has Kadac--Klee property, we can conclude that
$x_{n_k}\rightarrow\omega=\Pi_{F(S)\cap EP(f)}x.$ Hence $\{x_n\}$ converges strongly to $\Pi_{F(S)\cap EP(f)}x$
\end{proof}
 \begin{cor}
 Let $C$ be a nonempty, closed convex subset of uniformly smooth and uniformly convex Banach space $E$ and $S$ be a relatively nonexpansive self-mapping of $C$ with $F(S)\neq\phi$.
Assume that $0<a\leq\alpha_n\leq1$ and  $\{\beta_n\}$ is sequence
in $[0,1]$ such that $\liminf_{n\rightarrow\infty}\beta_n(1-\beta_n)>0$. If $\{x_n\}$ and $\{u_n\}$ be sequences generated by
$x=x_1\in C$ and
\begin{equation*}
\begin{cases}
 & z_n=J^{-1}(\beta_nJx_n+(1-\beta_n)JSx_n),\\
 &u_n\in E\;\text{such that}\quad\langle y-u_n,Ju_n-Jz_n\rangle\geq0,\quad\forall\: y\in E\\
 & y_n=J^{-1}(\alpha_nJz_n+(1-\alpha_n)Ju_n),\\
 &C_n=\lbrace v\in C:\phi(v,y_n)\leq\phi(v,x_n)\rbrace,\\
 &Q_n=\lbrace z\in C:\langle x_n-z,Jx_n-Jx\rangle\leq0\rbrace,\\
 &x_{n+1}=\Pi_{C_n\cap Q_n}x,
  \end{cases}
\end{equation*}
for all $n\in\mathbb{N}$. Then, $\{x_n\}$ converges strongly to $\Pi_{F(S)}x$, where $\Pi_{F(S)}$ is the generalized  projection of $E$ onto $F(S)$.
\end{cor}
\begin{proof}
Letting $f(x,y)=0$ for all $x,y\in E$ and $r_n=1$ for all $n \in
\mathbb{N}$, in Theorem \ref{thm1}, we get the desired result.
\end{proof}
\begin{cor}
 Let $C$ be a nonempty, closed convex subset of uniformly smooth and uniformly convex Banach space $E$ and $f$ be a bifunction from $C\times C$ to $\mathbb{R}$ satisfying $(A1)-(A4)$ and $S$ be a relatively nonexpansive self-mapping of $C$ with $F(S)\cap EP(f)\neq\phi$.
Assume that  $\{r_n\}\subset(0,\infty)$ satisfies
$\liminf_{n\rightarrow\infty}r_n>0$ and  $\{\beta_n\}$ is sequence
in $[0,1]$ such that $\liminf_{n\rightarrow\infty}\beta_n(1-\beta_n)>0$. If $\{x_n\}$ and $\{u_n\}$ be sequences generated by
$x=x_1\in C$ and
\begin{equation*}
\begin{cases}
& z_n=J^{-1}(\beta_nJx_n+(1-\beta_n)JSx_n),\\
&u_n\in E\;\text{such that}\quad f(u_n,y) +\frac{1}{r_n}\langle y-u_n,Ju_n-Jz_n\rangle\geq0,\quad\forall\: y\in E\\
 &C_n=\lbrace v\in C:\phi(v,z_n)\leq\phi(v,x_n)\rbrace,\\
 &Q_n=\lbrace z\in C:\langle x_n-z,Jx_n-Jx\rangle\leq0\rbrace,\\
 &x_{n+1}=\Pi_{C_n\cap Q_n}x,
  \end{cases}
\end{equation*}
for all $n\in\mathbb{N}$. Then, $\{x_n\}$ converges strongly to $\Pi_{F(S)\cap EP(f)}x$, where $\Pi_{F(S)\cap EP(f)}$ is the generalized projection of $E$ onto $F(S)\cap EP(f)$.
\end{cor}
\begin{proof}
Letting $\alpha_n=1$, in Theorem \ref{thm1}, we get the desired result.
\end{proof}
 \begin{cor}
 Let $C$ be a nonempty closed convex subset of uniformly smooth and uniformly convex Banach space $E$ and $f$ be a bifunction from $C\times C$ to $\mathbb{R}$ satisfying $(A1)-(A4)$ with $ EP(f)\neq\phi$. Assume that $0<a\leq\alpha_n\leq1$ and $\{r_n\}\subset(0,\infty)$ satisfies
$\liminf_{n\rightarrow\infty}r_n>0$ and  $\{\beta_n\}$ is sequence
in $[0,1]$ such that $\liminf_{n\rightarrow\infty}\beta_n(1-\beta_n)>0$. If $\{x_n\}$ and $\{u_n\}$ be sequences generated by
$x=x_1\in C$ and
\begin{equation*}
\begin{cases}
 &u_n\in E\;\text{such that}\quad f(u_n,y) +\frac{1}{r_n}\langle y-u_n,Ju_n-Jx_n\rangle\geq0,\quad\forall\: y\in E\\
 & y_n=J^{-1}(\alpha_nJx_n+(1-\alpha_n)Ju_n),\\
 &C_n=\lbrace v\in C:\phi(v,y_n)\leq\phi(v,x_n)\rbrace,\\
 &Q_n=\lbrace z\in C:\langle x_n-z,Jx_n-Jx\rangle\leq0\rbrace,\\
 &x_{n+1}=\Pi_{C_n\cap Q_n}x,
  \end{cases}
\end{equation*}
for all $n\in\mathbb{N}$. Then, $\{x_n\}$ converges strongly to $\Pi _{EP(f)}x$, where $\Pi_{ EP(f)}$ is the generalized  projection of $E$ onto $EP(f)$.
\end{cor}
\begin{proof}
Letting $S=I$ in Theorem \ref{thm1}, we get the desired result.
\end{proof}

{\footnotesize}


\begin{thebibliography}{99}
\bibitem {Agarwal}
R. P. Agarwal, Donal O'Regan and D. R. Sahu, {\it Fixed Point Theory
for Lipschitzian-type Mappings with Applications}, Springer 2009.
\bibitem{Alber}
Ya. I. Alber, {\it Metric and generalized projection operators in Banach spaces: Properties and applications,} in: A. G. Kartsatos (Ed.), Theory and Applications of Nonlinear Operators of Accretive and Monotone Type, Marcel Dekker, New York, (1996), pp. 15--50.
\bibitem{am1}
S. Alizadeh and F. Moradlou, {\it Strong convergence theorems for m-generalized hybrid mappings in Hilbert spaces,} to appear in Topol. Methods Nonlinear Anal.

\bibitem{am2}
S. Alizadeh and F. Moradlou, {\it A strong convergence theorem for
equilibrium problems and generalized hybrid mappings,} Mediterr. J.
Math. DOI 10.1007/s00009-014-0462-6.

\bibitem{b}
E. Blum and W. Oettli, {\it From optimization and variational inequalities to equilibrium problems, Mathematics Students,} {\bf 63} (1994), 123--145.
\bibitem{But1}
D. Butnariu, S. Reich and A. J. Zaslavski, {\it Asymptotic behavior of relatively nonexpansive operators in Banach spaces,} J. Appl. Anal. {\bf7} (2001) 151--174.
\bibitem{But2}
 D. Butnariu, S. Reich and A. J. Zaslavski, {\it Weak convergence of orbits of nonlinear operators in reflexive Banach
spaces,} Numer. Funct. Anal. Optim. {\bf24} (2003) 489--508.
\bibitem{Censor}
Y. Censor and S. Reich, {\it Iterations of paracontractions and firmly nonexpansive operators with applications to
feasibility and optimization,} Optimization {\bf37} (1996) 323--339.
\bibitem{chidume}
C. E. Chidume and S. A. Mutangadura, { \it An example on the Mann
iteration method for Lipschitz pseudocontractions,} Proc. Am. Math.
Soc. {\bf129} (2001) 2359-–2363.

\bibitem{a2}
L. C. Ceng and J. C. Yao, {\it A hybrid iterative scheme for mixed equilibrium problems and fixed point problems,} J. Comput. Appl.
Math. {\bf214} (2008) 186--201.

\bibitem{a3}
L. C. Ceng, S. Al-Homidan and Q. H. Ansari, J. C. Yao, {\it An iterative scheme for equilibrium problems and fixed point problems of strict pseudo-contraction mappings,} J. Comput. Appl. Math. {\bf223} (2009) 967--974.

\bibitem{a4}
P. L. Combettes and S. A. Hirstoaga, {\it Equilibrium programming in Hilbert spaces,} J. Nonlinear Convex Anal. {\bf6} (2005) 117--136.

\bibitem{a5}
S. D. Flam and A. S. Antipin, {\it Equilibrium progamming using proximal-link algolithms,} Math. Program. {\bf78} (1997) 29--41.
\bibitem{Genel}
A. Genel and J. Lindenstrass, {\it An example concerning fixed
points,} Israel J. Math. { \bf22} (1975) 81–-86.

\bibitem{e}
F. Giannessi, A. Maugeri and P. M. Pardalos, {\it Equilibrium Problems: Nonsmooth Optimization and Variational Inequality Models,} Kluwer Academics Publishers, Dordrecht, Holland, 2001.
\bibitem {ishikawa}
S. Ishikawa, {\it Fixed points by a new iteration method}, Proc.
Amer. Math. Soc. {\bf40} (1974)  147--150.


\bibitem{a9}
C. Jaiboon and P. Kumam, {\it A hybrid extragradient viscosity approximation method for solving equilibrium problems and fixed point problems of infinitely many nonexpansive mappings,} Fixed Point Theory Appl. (2009) Article ID 374815, 32 pages.

\bibitem{a10}
 C. Jaiboon and P. Kumam, H. W. Humphries, {\it Weak convergence theorem by an extragradient method for variational inequality, equilibrium and fixed point problems,} Bull. Malays. Math. Sci. Soc. (2) {\bf32} (2) (2009) 173--185.
\bibitem{Kam}
S. Kamimura and W. Takahashi, {\it Strong convergence of a proximal-type algorithm in a Banach space,} SIAM J. Optim. {\bf13} (2002) 938--945.

\bibitem{Kinder}
D. Kinderlehrer and G. Stampacchia, {\it An Introduction to Variational Inequalities and their Applications, }Academic Press, New York, 1980.


\bibitem {mann}
W. R. Mann,  {\it  Mean value methods in iteration,}
  Proc. Amer. Math. Soc.  {\bf 4} (1953) 506--510.

\bibitem{martinez} C. Martinez-Yanes and H. K. Xu,
{\it  Strong convergence of CQ method fixed point biteration
processes,} Nonlinear Anal. {\bf 64} (2006) 2400--2411.
\bibitem {Mat}
S. Matsushita and W. Takahashi, {\it A strong convergence for relatively nonexpansive mappings in a Banach space,} J. Approx. Theory {\bf134} (2005) 257--266.
\bibitem{a16}
A. Moudafi and  M. Thera, {\it Proximal and dynamical approaches to equilibrium problems, in: Lecture Note in Economics and Mathematical Systems,} vol. {\bf477}, Springer-Verlag, New York, (1999), pp. 187--201.

\bibitem {Nakajo}
K. Nakajo and W. Takahashi, {\it Strong convergence theorems for
nonexpansive mapping and nonexpansive mappings},  J. Math. Anal.
Appl. {\bf 279} (2003) 372--379.

\bibitem{a18}
S. Plubtieng and P. Kumam, {\it Weak convergence theorem for monotone mappings and a countable family of nonexpansive mappings,} J. Comput. Appl. Math. {\bf224} (2009) 614--621.

\bibitem{a20}
X. Qin, Y. J. Cho and S. M. Kang, {\it Convergence theorems of common elements for equilibrium problems and fixed point problems in Banach spaces,} J. Comput. Appl. Math. {\bf225} (2009) 20--30.
\bibitem{Reich}
S. Reich, {\it A weak convergence theorem for the alternating method with Bregman distance,} in: A.G. Kartsatos (Ed.), Theory and Applications of Nonlinear Operators of Accretive and Monotone Type, Marcel Dekker, New York, (1996), pp. 313--318.
\bibitem{a27}
A. Tada and W. Takahashi, {\it Weak and strong convergence theorems for a nonexpansive mappings and an equilibrium problem,} J. Optim. Theory Appl. {\bf133} (2007) 359--370.

\bibitem{a23}
S. Takahashi and W. Takahashi, {\it Viscosity approximation methods for equilibrium problems and fixed point problems in Hilbert spaces,} J. Math. Anal. Appl. {\bf331} (2007) 506--515.

\bibitem{taka2003}
W. Takahashi and M. Toyoda, { \it Weak convergence theorems for
nonexpansive mappings and monotone mappings,} J. Optim. Theory Appl.
{\bf118} (2003) 417–-428.

\bibitem{a24}
W. Takahashi and  K. Zembayashi, {\it Strong convergence theorems by hybrid methods for equilibrium problems and relatively nonexpansive mapping,} Fixed Point Theory Appl. (2008) Article ID 528476, 11 pages.
\bibitem{T4}
W. Takahashi, K. Zembayashi, {\it Strong and weak convergence theorems for equilibrium problems and relatively nonexpansive mappings in Banach spaces,} Nonlinear Anal. {\bf70} (2009) 45--57.

 \bibitem{T3}
W. Takahashi, {\it Introduction to Nonlinear and Convex Analysis},
Yokohoma Publishers, Yokohoma, 2009.

\bibitem {T5}
W. Takahashi and J. -C. Yao, {\it weak convergence theorems for generalized hybrid mappings in Banach spaces,}
 Journal of  Nonlinear Analysis and Optimazition. {\bf 2} (2011) No. 1, 147--158.

\bibitem{Tan}
K. K. Tan, H. K. Xu, {\it Approximating fixed points of nonexpansive mappings by the Ishikawa iteration process},
J. Math. Anal. Appl. {\bf 178} (2) (1993) 301--308.

 \bibitem{Xu}
 H. K. Xu, {\it Inequalities in Banach spaces with application,} Nonlinear Anal. {\bf16} (1991) 1127--1138.
\end{thebibliography}
\end{document}